\documentclass[12pt]{article}
\usepackage{amssymb,amsthm,hyperref,amsmath,bm,amsfonts}
\usepackage{arydshln}
\usepackage{verbatim}
\usepackage{graphicx}
\usepackage{float}

\usepackage{subfigure}
\usepackage{multirow}
\usepackage{subfigure}
\usepackage{color}
\usepackage{appendix}

\newtheorem{theorem}{Theorem}[section]
\newtheorem{lemma}[theorem]{Lemma}

\newtheorem{remark}{Remark}[section]
\numberwithin{equation}{section}

\title{Feedback boundary control of 2-D hyperbolic systems with relaxation}
\author{
Haitian Yang\thanks{E-mail: yht21@mails.tsinghua.edu.cn}\\
\small{\textit{Department of Mathematical Sciences, Tsinghua University, Beijing 100084, China.}}\\ \\
	Wen-An Yong\thanks{E-mail: wayong@tsinghua.edu.cn}\\
\small{\textit{Department of Mathematical Sciences, Tsinghua University, Beijing 100084, China.}}\\
\small{\textit{Yanqi Lake Beijing Institute of Mathematical Sciences and Applications, Beijing 101408, China.}}
}

\date{\today}
\begin{document}
\maketitle{}

 \begin{abstract}
This paper is concerned with boundary stabilization of two-dimensional hyperbolic systems of partial differential
equations. By adapting the
Lyapunov function previously proposed by the second author for linearized hyperbolic systems with relaxation structure, we derive certain control laws so that the corresponding solutions decay exponentially in time. The result is illustrated with
an application to water flows in open channels. 
 \end{abstract}

 \hspace{-0.5cm}\textbf{Keywords:}
 \small{Boundary stabilization,
 	Control laws,
 	Hyperbolic relaxation systems,
 	Structural stability condition,
 	2-D Saint-Venant equations.}\\


\section{Introduction}
We are interested in boundary stabilization of 
two-dimensional hyperbolic partial differential equations with relaxation in bounded domains. Such problems have recently attracted much interest in the mathematical and engineering community due to wide applications. Typical examples are
the Saint–Venant equations and related models for open channels \cite{coron2007strict,DEHALLEUX20031365,DIAGNE2012109,DIAGNE2017345,DIAGNE2016130,dos2008boundary,PRIEUR201844}, the Aw-Rascle equations for road traffic \cite{yu2019traffic,yu2022traffic}, gas dynamics \cite{gugat2011existence}, supply chains \cite{coron2012controllability} or heat exchanges \cite{xu2002exponential}.


As far as we know, the existing works mainly deal with the spatially one-dimensional problems. Many of these works are based on the backstepping method \cite{7366757,krstic2008boundary}. This method converts the original hyperbolic system to a target system via an elaborate Volterra transformation, the stabilization for the target system can be easily achieved through homogeneous boundary conditions, and then control laws for the original system are derived according to the transformation. Another approach is the Lyapunov function method \cite{bastin2016stability}. This method can guide us to design proper boundary conditions (control laws) which ensure the exponential time-decay of the corresponding solutions. It is valid for hyperbolic partial differential equations coupled through source terms in certain special ways \cite{DIAGNE2012109}, while the backstepping method can handle arbitrarily strong coupling systems \cite{7366757}. For state-of-the-art developments of these approaches, we refer to \cite{AURIOL2016300,DEUTSCHER201754,DIAGNE2017345,doi:10.1137/15M1012712,yu2022traffic} for the backstepping method, and to \cite{BASTIN201766,8706524,coron2022lyapunov,HAYAT201952} for the Lyapunov function method.

However, multi-dimensional problems are more realistic and therefore cannot be evaded. For example, consider water flows in a river which can be regarded as a two-dimensional domain. An illustration of the watercourse between two sluice gates is shown in Figure \ref{figure1}. Spillways can be constructed on river banks. These hydraulic facilities serve to regulate flow rates or depth of water. The sluice gates, spillways and river banks constitute the boundary of the two-dimensional domain and the regulations can be considered as boundary conditions.

\begin{figure}
	\begin{center}
		\includegraphics[width=7cm]{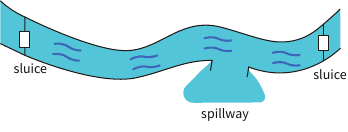}     
		\label{figure1}
		\caption{Watercourse}                             
	\end{center}                                 
\end{figure}

In attempting to study the multi-dimensional problems, we immediately encounter two difficulties. First, the aforementioned methods heavily rely on the diagonalizable feature of one-dimensional hyperbolic systems, and hence seem to be invalid in multi-dimensional cases. Second, the boundary for the multi-dimensional problems is much more complicated than that for the one-dimensional case where it usually consists of only two endpoints. 

In contrast, the recently proposed Lyapunov function in \cite{YONG2019252} can be easily generalized to the multi-dimensional hyperbolic systems satisfying the structural stability condition \cite{yong1999singular}.  As shown in \cite{yong47basic,yong2008interesting},
this class of systems covers many physically relevant equations, although it is not as general as that studied in \cite{7366757}.
By the way, we refer to  \cite{HERTY201612,WANG2020104815} for other recent works exploiting the structural stability condition to achieve boundary stabilization in the one-dimensional case.

In this paper, we generalize the results in \cite{YONG2019252} to achieve boundary stabilization for the two-dimensional problems. As an application, we derive certain control laws for the two-dimensional Saint-Venant equations with small river velocity. 

In comparison, other multi-dimensional works known to us are \cite{xu2002exponential} and \cite{herty2022stabilization}. In \cite{xu2002exponential}, boundary stabilization was achieved for dissipative symmetric hyperbolic systems under certain ad hoc assumptions (see formulas (2.28) and (2.32) therein). In \cite{herty2022stabilization} a key assumption is that the multiple coefficient matrices corresponding to the spatial variables can be diagonalized simutaneously. These assumptions limit the applicability of their results in \cite{herty2022stabilization,xu2002exponential}.

\section{Preliminaries}  \label{se2}
Consider a linear system of first-order partial differential equations:
\begin{equation} \label{1}	\tilde{U}_t+\tilde{A}_1\tilde{U}_x+\tilde{A}_2\tilde{U}_y=\tilde{Q}\tilde{U}
\end{equation}
defined on $(t,x,y) \in [0,\infty ) \times \Omega.$ Here, $\tilde{U}=\Tilde{U}(t,x,y)\in \mathbb{R}^N$ is the unknown, $\tilde{A}_1,\tilde{A}_2$ and $\tilde{Q}$ are given constant  matrices in $\mathbb{R}^{N\times N}$, and $\Omega$ is a bounded domain with a piecewise smooth boundary. A typical example is polygon.

For the above system, we refer to \cite{yong47basic,yong2008interesting} and assume that it satisfies the structural stability condition proposed in \cite{yong1999singular}. Namely, there exists an invertible matrix $\tilde{P} \in \mathbb{R}^{N \times N}$ such that
\begin{equation} \label{2}
	\tilde{P}\tilde{Q}\tilde{P}^{-1}=-\begin{pmatrix}
		0 & 0 \\
		0 & e
	\end{pmatrix}
\end{equation}
with $e \in \mathbb{R}^{r \times r}$ invertible $(0<r<N)$;
there exists a symmetric positive definite matrix $\tilde{A}_0 \in \mathbb{R}^{N \times N}$ such that both $\tilde{A}_0\tilde{A}_1$ and $\tilde{A}_0\tilde{A}_2$ are symmetric; and 
\begin{equation*}
	\tilde{A}_0\tilde{Q}+\tilde{Q}^T\tilde{A}_0 \leq -\tilde{P}^T \begin{pmatrix}
		0 & 0 \\
		0 & I_r
	\end{pmatrix}\tilde{P}.
\end{equation*}
Here and below, the superscript $^T$ denotes the transpose of the matrix and  $I_r$ stands for the unit matrix of order $r$. Note that the existence of the symmetrizer $\tilde{A}_0$ ensures the hyperbolicity of system (\ref{1}). \\

\begin{remark}
	As shown in \cite{yong47basic,yong1999singular,yong2008interesting}, the structural stability condition has its root in non-equilibrium thermodynamics and is fulfilled by many classical models from mathematical physics. Examples occur in kinetic theories (both moment closure systems and discrete velocity models), gas dynamics with damping or with relaxation, chemically reactive flows, radiation hydrodynamics, relativistic dissipative fluid dynamics, magnetohydrodynamics, nonlinear optics, traffic flows, river flows \cite{yong47basic,yong1999singular,yong2008interesting}, compressible viscoelastic fluid dynamics \cite{yong2014newtonian}, probability theory and axonal transport \cite{YONG2019252}, certain chemical exchange processes \cite{xu2002exponential}, and so on. Namely, the stability condition above defines a wide class of physically relevant problems, although this class is not as general as that studied in \cite{7366757}. Therefore, it is reasonable to study such hyperbolic systems under the structural stability condition.
	
\end{remark}

Under the stability condition, it was proved in \cite{yong1999singular} that $A_0=\tilde{P}^{-T}\tilde{A}_0\tilde{P}^{-1}$ is a block-diagonal matrix with the same partition as in (\ref{2}).

Set $U=\tilde{P}\tilde{U},A_1=\tilde{P}\tilde{A}_1\tilde{P}^{-1}$ and $A_2=\tilde{P}\tilde{A}_2\tilde{P}^{-1}.$ It follows from (\ref{1}) that
\begin{equation} \label{3}
	U_t+A_1U_x+A_2U_y=-\begin{pmatrix}
		0\ 0 \\ 0 \ e
	\end{pmatrix}U,
\end{equation}
or 
\begin{equation} \label{4}
	\begin{aligned}
		u_t+a_1u_x+b_1q_x+a_2u_y+b_2q_y&=0,\\[4mm]
		q_t+c_1u_x+d_1q_x+c_2u_y+d_2q_y&=-eq.
	\end{aligned}
\end{equation}  
Here $\ u=u(t,x,y)  :  [0,\infty ) \times \Omega \rightarrow \mathbb{R}^{N-r} $ , $ q=q(t,x,y)  :  [0,\infty ) \times \Omega \rightarrow \mathbb{R}^{r},$
and $a_i,b_i,c_i,d_i(i=1,2)$ are matrices with proper dimensions; $U=\begin{pmatrix} u \\ q \end{pmatrix}, A_1=\begin{pmatrix}
	a_1 & b_1 \\
	c_1 & d_1
\end{pmatrix}$ and $A_2=\begin{pmatrix}
	a_2 & b_2 \\
	c_2 & d_2
\end{pmatrix}.$ 

For system (\ref{3}) or (\ref{4}), it is not difficult to see that the structural stability condition is composed of the following two items.
\begin{enumerate} 
	\item [(i)] There exists a symmetric positive definite matrix $A_0=\begin{pmatrix}
		X_1 &  0 \\ 0 & X_2  
	\end{pmatrix}$, with $X_1 \in \mathbb{R}^{(N-r) \times (N-r)}, X_2 \in \mathbb{R}^{r \times r}$ , such that both $A_0A_1$ and $A_0A_2$ are symmetric;
	
	\item [(ii)]   $X_2e+e^TX_2 $ is positive definite.
\end{enumerate}

Moreover, we assume that 
\begin{enumerate} 
	\item [(iii)]	 There exist real numbers $\alpha$ and $\beta$ so that the $(N-r) \times (N-r)$-matrix $\alpha a_1+\beta a_2$ has only negative eigenvalues. \\
\end{enumerate} 

\begin{remark}
	Assumption (iii) is a generalization of the assumption (A3) in \cite{YONG2019252}. It is technical and might be replaced by other weaker assumptions. Note that the structural stability condition ensures that $X_1(\alpha a_1+ \beta a_2)$ is symmetric and therefore $\alpha a_1+ \beta a_2$ has only
	real eigenvalues. A special case is $a_1$(or $a_2$) only has negative(or positive) eigenvalues, this case is the assumption (A3) in \cite{YONG2019252}. 
\end{remark}

Next, denote by $\mathbf{n}(x,y)=(n_1(x,y),n_2(x,y))$  the unit outward normal vector at boundary point $(x,y) \in \partial \Omega.$ Thanks to the hyperbolicity, the matrix $[n_1(x,y)A_1+n_2(x,y)A_2]$ can be diagonalized, i.e. there exists an invertible matrix $P=P(x,y)$ such that
\begin{equation} \label{5}
	\begin{aligned}
		\Lambda=\Lambda(x,y)&=P^{-1}[n_1(x,y)A_1+n_2(x,y)A_2]P\\& =\begin{pmatrix}
			\Lambda_+(x,y) & 0 & 0 \\
			0  & 0 & 0 \\
			0 & 0 & \Lambda_-(x,y)
		\end{pmatrix}.
	\end{aligned}
\end{equation}
Here, $\Lambda_+(x,y) \in \mathbb{R}^{p \times p}$ and $\Lambda_-(x,y)  \in \mathbb{R}^{n \times n}$ are both diagonal; they contain positive and negative eigenvalues of $[n_1(x,y)A_1+n_2(x,y)A_2],$ respectively; $p=p(x,y)$ and $n=n(x,y)$ are the numbers of the positive and negative eigenvalues, respectively. Note that $p$ and $n$ depend on the boundary point $(x,y).$ Remarkably, we do not assume that the boundary is non-characteristic.

With the notations just defined, we formulate the following fact. It is similar to Lemma 3.3 in \cite{HERTY201612} and exposes an important relation between the transformation matrix $P=P(x,y)$ in (\ref{5}) and the symmetrizer $A_0$ in the assumption (i).
\\

\begin{lemma} \label{l2}
	Under assumption (i), there exist three symmetric positive definite matrices $X_+=X_+(x,y) \in \mathbb{R}^{p\times p}, X_-=X_-(x,y) \in \mathbb{R}^{n\times n},X_0=X_0(x,y) \in \mathbb{R}^{(N-p-n)\times (N-p-n)}$ such that 
	\begin{equation*}
		P^{T}(x,y)A_0P(x,y)=\begin{pmatrix}
			X_+(x,y) & 0  & 0\\
			0 & X_0(x,y) & 0 \\
			0 & 0 & X_-(x,y)
		\end{pmatrix}.
	\end{equation*}
\end{lemma}
\begin{proof}
	Fix the boundary point $(x,y)$ and the dependence on $(x,y)$ may be omitted throughout this proof. Observe that 
	$$
	P^{-1}[n_1 A_1+n_2 A_2]P=\Lambda=P^T[n_1 A_1+n_2 A_2]^TP^{-T},
	$$
	where $\Lambda$ is given in (\ref{5}). Since $A_0(n_1 A_1+n_2 A_2)=(n_1 A_1+n_2 A_2)^TA_0,$ we have $P^TA_0P\Lambda=\Lambda P^TA_0 P.$ Namely, the diagonal matrix $\Lambda$ commutes with the symmetric matrix $P^TA_0P.$ Thus the latter is of block-diagonal with $X_+,X_-$ and $X_0$ of proper dimensions. Since $A_0$ is positive definite, so are $X_+,X_-$ and $X_0$.
\end{proof}

Furthermore, with $P=P(x,y)$ defined in (\ref{5}) we introduce
\begin{equation} \label{6}
	\zeta(t,x,y)=\begin{pmatrix}
		\zeta_+(t,x,y) \\
		\zeta_0(t,x,y) \\
		\zeta_-(t,x,y)
	\end{pmatrix}:=P^{-1}(x,y)U(t,x,y).
\end{equation}
Here $\zeta_+=\zeta_+(t,x,y) \in \mathbb{R}^p,\zeta_-=\zeta_-(t,x,y) \in \mathbb{R}^n$ and $\zeta_0=\zeta_0(t,x,y) \in \mathbb{R}^{N-p-n}$ correspond to the partition in (\ref{5}).
According to \cite{10.1093/acprof:oso/9780199211234.001.0001,doi:10.1137/1020095}, the boundary conditions should assign the incoming variables $\zeta_-$ in terms of the outgoing variables $\zeta_+$ at each boundary point. For properly given boundary conditions and initial data, the existence and uniqueness of $L^2$-solutions $U=U(t,x,y)$ for system (\ref{3}) can be found in \cite{10.1093/acprof:oso/9780199211234.001.0001} and the references therein.

\section{Main results} \label{se3}

Inspired by \cite{YONG2019252}, we refer to the assumptions (i)-(iii) and consider the following Lyapunov function
\begin{equation} \label{7}
	L(U(t))= \iint_{\Omega} \lambda(x,y) U^T(t,x,y) A_0 U(t,x,y) dxdy
\end{equation}
with 
$$
\lambda(x,y)=K+\alpha x+\beta y>0, \quad  \forall (x,y) \in \overline{\Omega},
$$ 
where $K$ is a positive constant. Such a constant $K$ is available for bounded domain $\Omega$ and will be chosen in the proof of Lemma \ref{l1} below. This Lyapunov function takes advantage of the
structural stability condition. It is obvious that $L(U(t))$ is equivalent to 
$$
\| U(t) \|^2=\| U(t,\cdot) \|^2:=\int_{\Omega} U^T(t,x,y) U(t,x,y) dxdy.
$$
Our main result is \\

\begin{theorem} \label{t1}
	Under the assumptions (i),(ii) and (iii), there exist proper boundary conditions such that the corresponding solutions $U=U(t,x,y)$ to system (\ref{3}) with initial data $U_0=U_0(x,y)\in L^2(\Omega)$ are exponentially stable. Namely, there exist positive constants $\nu$ and $C$ such that 
	\begin{equation*}
		\| U(t) \| \leq C\exp(-\nu t) \| U_0 \|.
	\end{equation*}
\end{theorem}

By boundary stabilization, we mean that choosing proper boundary conditions ensures the exponential stability.  The proof of this theorem relies on the following lemma. \\

\begin{lemma} \label{l1}
	Under the condition of Theorem \ref{t1}, there exists a positive constant $\nu$ such that
	\begin{equation*}
		\frac{d}{dt} L(U(t))+\mathcal{BC} \leq -\nu L(U(t)).
	\end{equation*}
	Here
	\begin{equation} \label{8}
		\mathcal{BC}=\int_{\partial \Omega} \lambda(x,y) U^TA_0(n_1A_1+ n_2A_2)U  d\ell,
	\end{equation}
	where $\mathbf{n}=(n_1,n_2)$ is the outward unit normal vector defined before and $d\ell$ denotes the Lebesgue measure of the boundary $\partial \Omega.$
\end{lemma}

\begin{proof}
	Under assumptions (i),(ii) and (iii), we multiply the system (\ref{3}) or (\ref{4}) with $\lambda(x,y)U^TA_0$ from the left and use the expression $\lambda(x,y)=K+\alpha x+\beta y$ to obtain 
	\begin{equation*}
		\begin{aligned}
			&(\lambda(x,y)U^TA_0U)_t+(\lambda(x,y)U^TA_0A_1U)_x \\[4mm] &  +	 (\lambda(x,y)U^TA_0A_2U)_y \\[4mm] 
			=&\alpha U^TA_0A_1U+\beta U^TA_0A_2U -\lambda(x,y)q^T(X_2e+e^TX_2)q. 
		\end{aligned}
	\end{equation*}
	Notice that both $-(X_2e+e^TX_2)$ and $X_1(\alpha a_1 + \beta a_2)$
	are symmetric and strictly negative definite. There exists a large constant
	$\chi > 0$ such that
	\begin{equation*}
		\begin{aligned}
			& \quad \ \alpha U^TA_0A_1U+\beta U^TA_0A_2U \\[4mm] 
			&=u^TX_1(\alpha a_1+\beta 
			a_2)u+q^TX_2(\alpha c_1+ \beta c_2)u \\[4mm]  & \quad +u^TX_1(\alpha b_1+ \beta b_2) q+q^TX_2(\alpha d_1 + \beta d_2) q \\[4mm] 
			&\leq \frac{1}{2}u^TX_1(\alpha a_1+ \beta a_2)u+\chi q^T(X_2e+e^TX_2)q.
		\end{aligned}
	\end{equation*}
	Now we choose $K$ in the expression of $\lambda(x,y)$ so that
	\begin{equation*}
		\lambda(x,y)-\chi>\frac{1}{2}\lambda(x,y)>0, \quad \forall (x,y) \in \overline{\Omega}.
	\end{equation*}
	With such a $K$, we have 
	\begin{equation*}
		\begin{aligned}
			&(\lambda(x,y)U^TA_0U)_t+(\lambda(x,y)U^TA_0A_1U)_x \\[4mm]  &  +	 (\lambda(x,y)U^TA_0A_2U)_y   \\[4mm] 
			\leq & \frac{1}{2}u^TX_1(\alpha a_1+\beta a_2)u-\frac{1}{2}\lambda(x,y) q^T(X_2e+e^TX_2)q.
		\end{aligned}
	\end{equation*}
	Denote by $\sigma$ the minimum of smallest eigenvalues of the symmetric positive definite matrices $-X_1(\alpha a_1+\beta a_2)$ and $\lambda(x,y)(X_2e+e^TX_2)$ over $(x,y) \in \Omega.$ We integrate the last inequality over $(x,y) \in \Omega$ and use Green's formula to get
	\begin{equation*}
		\frac{d}{dt} L(U(t))+\mathcal{BC} \leq -\frac{\sigma}{2} \| U(t)\|^2 \leq -\frac{\sigma}{2\lambda_{\max}\rho(A_0)}L(U(t)),
	\end{equation*}
	with $\mathcal{BC}$ the boundary term defined in (\ref{8}), where $\lambda_{\max}$ is the supremum of $\lambda(x,y)$ in $\bar{\Omega}$ and $\rho(A_0)$ is the maximum eigenvalue of $A_0.$ Hence the lemma follows with
	\begin{equation*}
		\nu=\frac{\sigma}{2\lambda_{\max}\rho(A_0)}
	\end{equation*}
	and the proof is complete.
\end{proof} 

Thanks to Lemma \ref{l1}, Theorem \ref{t1} follows immediately provided that $\mathcal{BC} \geq 0$. In view of Lemma \ref{l2} and formulas (\ref{5}) and (\ref{6}), the boundary term $\mathcal{BC}$ can be written as
\begin{equation*}
	\begin{aligned}
		\mathcal{BC}&=\int_{\partial \Omega} \lambda(x,y)U^TA_0(n_1A_1+ n_2A_2)U  d\ell \\
		&=\int _{\partial \Omega} \lambda(x,y)(P^{-1}U)^T\begin{pmatrix}
			X_+\Lambda_+ & 0&0 \\
			0 & 0 & 0 \\
			0 & 0 & X_-\Lambda_-
		\end{pmatrix}P^{-1}U  d\ell \\
		&=\int _{\partial \Omega} \lambda(x,y)\zeta_+^T
		X_+\Lambda_+ \zeta_+ d\ell \\& \quad + \int _{\partial \Omega} \lambda(x,y)\zeta_-^T X_-\Lambda_- \zeta_- d\ell.
	\end{aligned}
\end{equation*}
Recall that $X_+$ and $\Lambda_+$ are symmetric positive definite, and $X_+\Lambda_+$ is symmetric, thus $X_+\Lambda_+$ is also positive definite. Similarly, $X_-\Lambda_-$ is negative definite. Hence we have  $$\int _{\partial \Omega} \lambda(x,y) \zeta_+^TX_+\Lambda_+\zeta_+ d\ell \geq 0$$ and $$\int _{\partial \Omega} \lambda(x,y)\zeta_-X_-\Lambda_-\zeta_- d\ell \leq 0.$$ Thus the theorem is proved if
\begin{equation*} 
	\int _{\partial \Omega} \lambda(x,y)\zeta_+^T
	X_+\Lambda_+ \zeta_+ d\ell \geq -\int _{\partial \Omega} \lambda(x,y)\zeta_-^T X_-\Lambda_- \zeta_- d\ell.
\end{equation*}
This inequality holds true at least in the trivial case that the incoming variables $\zeta_-$ are chosen to be zero.

\section{Applications to the Saint-Venant equations} \label{se5}
In the previous section, we have shown that the boundary stabilization can be achieved for systems satisfying the assumptions (i)-(iii) and defined on a bounded domain. Here, we present workable boundary conditions or control laws for hydraulic systems described by the Saint-Venant equations in two dimensions.

For the sake of simplicity, we consider a rectangular open channel and ignore the turbulent diffusion and wind stress. Denote the rectangular channel by $(0,L) \times (0,1)$ with $L$ a large positive constant for the distance between two sluices in the channel. The hydraulic system is characterized by the water depth $H=H(t,x,y)$, the  velocity $W=W(t,x,y)$ in the $x$-direction and the velocity $V=V(t,x,y)$ in the $y$-direction. 

When the velocities are small, the fluid resistance is linearly proportional to the velocity. In this situation, the dynamics of the hydraulic system is described by the  Saint-Venant equations
\begin{equation*}
	\begin{aligned}
		&\frac{\partial H}{\partial t} + W \frac{\partial H}{\partial x}+H \frac{\partial W}{\partial x}+V \frac{\partial H}{\partial y}+H \frac{\partial V}{\partial y}=0, \\
		&\frac{\partial W}{\partial t} + g\frac{\partial H}{\partial x}+W\frac{\partial W}{\partial x}+V\frac{\partial W}{\partial y}=gS_x-kW+lV, \\
		&\frac{\partial V}{\partial t} +W\frac{\partial V}{\partial x}+g\frac{\partial H}{\partial y}+ V\frac{\partial V}{\partial y}=gS_y-lW-kV.
	\end{aligned}
\end{equation*}
This is slightly different from the two-dimensional Saint-Venant equations in \cite{LICH1992498} because the velocities are assumed to be small.
Here $g$ is the gravity constant, $S_x$ is the bottom slope
of the channel in the $x$-direction, $S_y$ is the bottom slope
of the channel in the $y$-direction, $k$ is the viscous drag coefficient, $l$ is the Coriolis coefficient associated with the Coriolis force, and $k$ and $l$ are positive constants. 

For this system, a steady state is a constant state $(H^*,W^*,V^*)$ satisfying the conditions
\begin{equation*}
	gS_x=kW^*-lV^*, \quad 
	gS_y=kV^*+lW^*.
\end{equation*}
Define  
\begin{equation*}
	\begin{aligned}
		h(t,x,y)&=H(t,x,y)-H^*, \\
		w(t,x,y)&=W(t,x,y)-W^*, \\
		v(t,x,y)&=V(t,x,y)-V^*,
	\end{aligned}
\end{equation*} 
as deviations of the steady state.
Then linearizing the Saint-Venant equations around the steady state yields
\begin{equation} \label{9}
	\frac{\partial}{\partial t} \begin{pmatrix}
		h \\ w \\ v
	\end{pmatrix}+\bar{A}_1	\frac{\partial}{\partial x} \begin{pmatrix}
		h \\ w \\ v
	\end{pmatrix}+	\bar{A}_2\frac{\partial}{\partial y} \begin{pmatrix}
		h \\ w \\ v
	\end{pmatrix} = B \begin{pmatrix}
		h \\ w \\ v
	\end{pmatrix}
\end{equation} 
with 
\begin{equation*}
	\begin{aligned}
		\bar{A}_1&=\begin{pmatrix}
			W^* & H^* & 0 \\ g & W^* & 0 \\ 0 & 0 & W^*
		\end{pmatrix}, \bar{A}_2=\begin{pmatrix}
			V^* & 0 & H^* \\ 0 & V^* & 0 \\ g & 0 & V^*
		\end{pmatrix}, \\ B&=\begin{pmatrix}
			0 & 0 & 0 \\ 0 & -k & l \\ 0 & -l & -k
		\end{pmatrix}.
	\end{aligned}
\end{equation*}
Set $\tilde{h}=\sqrt{\frac{g}{H^*}} h$. Then system (\ref{9}) is equivalent to its symmetric version
\begin{equation} \label{10}
	\frac{\partial}{\partial t} \begin{pmatrix}
		\tilde{h} \\ w \\ v
	\end{pmatrix}+A_1	\frac{\partial}{\partial x} \begin{pmatrix}
		\tilde{h} \\ w \\ v
	\end{pmatrix}+	A_2\frac{\partial}{\partial y} \begin{pmatrix}
		\tilde{h} \\ w \\ v
	\end{pmatrix} = B \begin{pmatrix}
		\tilde{h} \\ w \\ v
	\end{pmatrix}
\end{equation} 
with 
\begin{equation*}
	\begin{aligned}
		A_1&=\begin{pmatrix}
			W^* & \sqrt{gH^*} & 0 \\ \sqrt{gH^*} & 	W^*  & 0 \\ 0 & 0 & 	W^* 
		\end{pmatrix}, A_2=\begin{pmatrix}
			V^* & 0 & 	\sqrt{gH^*}  \\ 0 & V^* & 0 \\ \sqrt{gH^*} & 0 & V^*
		\end{pmatrix}, \\ B&=\begin{pmatrix}
			0 & 0 & 0 \\ 0 & -k & l \\ 0 & -l & -k
		\end{pmatrix}.
	\end{aligned}
\end{equation*}

Observe that $A_1$ and $A_2$ are symmetric, and
\begin{equation*}
	B+B^T=\begin{pmatrix}
		0 & 0 & 0 \\ 0 &	-2k & 0 \\ 0 & 0 & -2k
	\end{pmatrix}.
\end{equation*}
Therefore the system (\ref{10}) satisfies assumption (i) and (ii) with $A_0=I_3$. Furthermore, the assumption (iii) holds if 
$$
(W^*,V^*)\neq (0,0),
$$
which is particularly true if $W^*>0$ and $V^*=0.$ This means that the steady water flows only in the $x$-direction.

Assume
$$
0<W^*<\sqrt{gH^*},\qquad  V^*=0, 
$$
which is consistent with the previous assumption that the water velocity is small. Thus we can choose the weighted function 
$\lambda(x,y)$ in (\ref{7}) to be $\lambda(x,y)=2L-x$.  In this case, the boundary term (\ref{8}) for the system (\ref{10}) becomes
\begin{equation} \label{11}
	\begin{aligned}
		&\mathcal{BC}
		=
		-2\sqrt{gH^*}\int_{0}^{L} [(2L-x) (\tilde{h}v)]\big|_{y=0}dx\\&\qquad \ +2\sqrt{gH^*}\int_{0}^{L}  [(2L-x)(\tilde{h}v)]\big|_{y=1} dx \\
		&+LW^*\int_0^1 [v^2+\tilde{h}^2+2\frac{\sqrt{gH^*}}{W^*}\tilde{h}w+w^2]\big|_{x=L}dy \\
		&-2LW^*\int_0^1 [v^2+\tilde{h}^2+2\frac{\sqrt{gH^*}}{W^*}\tilde{h}w+w^2]\big|_{x=0} dy.
	\end{aligned}
\end{equation}

To proceed, we refer to \cite{chanson2004hydraulics,DIAGNE2012109} and make the following assumptions
\begin{enumerate}
	\item [(a)] The river banks are solid walls except a spillway constructed in the lower bank.
	
	\item [(b)] Sluice gates are located at the left and right edges.
	
	\item [(c)]  Both sluice gates and spillway can be used to regulate the water velocities and to measure the water level.
\end{enumerate}

Assumption (a) implies that the normal velocity $V^*+v=v=0$ at the banks except the spillway. Suppose the spillway is located in $\{(x,y)\in (\frac{L}{3},\frac{2L}{3})\times \{0\}\}.$  Then the first two integrals in the boundary term (\ref{11}) become 
$$
-2\sqrt{gH^*}\int_{L/3}^{2L/3} [(2L-x) (\tilde{h}v)]\big |_{y=0} dx.
$$

Thanks to Assumptions (b) and (c), we regulate the velocities $w(t,0,y)$ and $w(t,L,y)$ so that
\begin{align}
	\label{12}   &\int_0^1 \tilde{h}w|_{x=L}dy + \frac{W^*}{2\sqrt{gH^*}}\int_0^1 \tilde{h}^2\big|_{x=L}dy \geq 0, \\
	\label{13}  &\int_0^1 \tilde{h}w|_{x=0}dy  + \frac{W^*}{2\sqrt{gH^*}}\int_0^1 (\tilde{h}^2+w^2)\big|_{x=0}dy \leq 0
\end{align}
according to the measured water levels $h(t,0,y),h(t,L,y)$ at sluice gates. Then the last two integrals in (\ref{11}) are greater than
$$
-2LW^* \int_{0}^{1} v^2|_{x=0} dy.
$$

Hence the boundary term $\mathcal{BC}\geq 0$ provided that
\begin{equation} \label{14}
	\begin{aligned}
		\sqrt{gH^*}\int_{L/3}^{2L/3} [(2L-x) (\tilde{h}v)]\big |_{y=0} dx&\\ +L W^*\int_0^1v^2\big |_{x=0}dy&\leq 0.
	\end{aligned}
\end{equation}
Thanks to Assumption (c), we regulate the flow velocity $v(t,x,0)$ at the spillway and $v(t,0,y)$ at the left sluice gate such that the inequality (\ref{14}) holds. In this way, boundary stabilization can be achieved for the system.

Now we derive boundary control laws for the system. By computing the matrix $P=P(x,y)$ in (\ref{5}), we can find the incoming and outgoing variables at each boundary point $(x,y).$ At the upper bank, the incoming variable is $\tilde{h}-v$ and the outgoing variable is $\tilde{h}+v.$ Thus the solid wall boundary condition $v(t,x,1)=0$ can be written as the classic form of boundary conditions \cite{higdon1986initial} for the first-order hyperbolic system (\ref{10}):
$$
[\tilde{h}-v](t,x,1)=[\tilde{h}+v](t,x,1), \qquad x \in (0,L).
$$
Similarly, for the lower bank the boundary condition is  
$$
[\tilde{h}+v](t,x,0)=[\tilde{h}-v](t,x,0), \qquad x \in (0,\frac{L}{3}) \cup (\frac{2L}{3},L).
$$

At the right sluice gate, the incoming variable is $\tilde{h}-w$ and the outgoing variables are $\tilde{h}+w$ and $v$. We refer to the classic theory \cite{higdon1986initial} and specify the boundary condition as
\begin{equation} \label{15}
	[\tilde{h}-w](t,L,y)=k_1[\tilde{h}+w](t,L,y), \quad y \in (0,1) 
\end{equation}
with $k_1$ a constant. With this boundary condition, the left-hand side of (\ref{12}) becomes
$$
\int_0^1(\frac{1-k_1}{1+k_1}+\frac{W^*}{2\sqrt{gH^*}})\tilde{h}^2|_{x=L}dy.
$$
Therefore the inequality (\ref{12}) holds for $k_1 \in (-1,1).$

At the left sluice gate, the incoming variables are $\tilde{h}+w$ and $v$, and the outgoing variable is $\tilde{h}-w.$ The classic theory \cite{higdon1986initial} suggests to give $\tilde{h}+w$ and $v$ in terms of $\tilde{h}-w.$ For $\tilde{h}+w$ we take
\begin{equation} \label{16}
	[\tilde{h}+w](t,0,y)=k_2[\tilde{h}-w](t,0,y), \quad y \in (0,1), 
\end{equation}
with $k_2$ to be determined. With this, the left-hand side of (\ref{13}) becomes
$$
\int_0^1\left[\frac{W^*}{2\sqrt{gH^*}}(\frac{k_2-1}{k_2+1})^2+\frac{k_2-1}{k_2+1}+\frac{W^*}{2\sqrt{gH^*}}\right]  \tilde{h}^2|_{x=0}dy.
$$
Note that the quadratic function of $ \frac{k_2-1}{k_2+1}$ in the above integral:
$$
\frac{W^*}{2\sqrt{gH^*}}(\frac{k_2-1}{k_2+1})^2+\frac{k_2-1}{k_2+1}+\frac{W^*}{2\sqrt{gH^*}}
$$
has two real roots due to $W^*<\sqrt{gH^*}.$ Thus the inequality (\ref{13}) holds true if $k_2$ is chosen so that $\frac{k_2-1}{k_2+1}$ is between the two roots. The incoming variable $v(t,0,y)$ will be given below for the inequality (\ref{14}).

For the spillway, the incoming variable is $\tilde{h}+v$ and the outgoing variable is $\tilde{h}-v$. As before, we assign
\begin{equation} \label{17}
	[\tilde{h}+v](t,x,0)=k_3 [\tilde{h}-v](t,x,0) , \quad x \in (\frac{L}{3},\frac{2L}{3}),
\end{equation}
with $k_3$ a constant. Moreover, we specify the incoming variable $v(t,0,y)$ at the left sluice gate as
\begin{equation} \label{18}
	v(t,0,y)=k_4[\tilde{h}-v](t,\frac{L}{3}(y+1),0), \quad y \in (0,1),
\end{equation}
with $k_4$ a constant. Note that $\frac{L}{3}(y+1)\in (\frac{L}{3},\frac{2L}{3})$ for $y \in (0,1).$ The last two equalities give
\begin{equation*}
	\begin{aligned}
		&v(t,x,0)=\frac{k_3-1}{k_3+1}\tilde{h}(t,x,0), \quad  x \in (\frac{L}{3},\frac{2L}{3}), \\
		&v(t,0,y)=\frac{2k_4}{k_3+1}\tilde{h}(t,\frac{L}{3}(y+1),0), \quad y \in (0,1).
	\end{aligned}
\end{equation*}
With these, the left-hand side of the inequality (\ref{14}) can be written as
\begin{equation*}
	\begin{aligned}
		&LW^*\int_0^1  \frac{4k_4^2}{(k_3+1)^2}\tilde{h}^2(t,\frac{L}{3}(y+1),0) dy \\ + &\frac{\sqrt{gH^*}L^2}{9}\int_0^1 (5-y)\frac{k_3-1}{k_3+1} \tilde{h}^2(t,\frac{L}{3}(y+1),0) dy.
	\end{aligned}
\end{equation*}
This is non-positive if $k_3 \in (-1,1)$ and $|k_4|$ is small enough. 

In conclusion, the boundary conditions (\ref{15})-(\ref{18}) with properly chosen parameters $k_i \ (i=1,2,3,4)$ ensure that the boundary term $\mathcal{BC}$ is non-negative. Consequently, boundary stabilization is achieved for the two-dimensional linearized Saint-Venant equations.

\section{Concluding remarks}
In this paper, we propose a framework for boundary stabilization of two-dimensional first-order hyperbolic systems with relaxation. As an example, a set of control laws (\ref{15})-(\ref{18}) are chosen as boundary conditions for the Saint-Venant equations. 

Clearly, the choices of boundary conditions are numerous, especially for two-dimensional problems.
For example, the incoming variable at the spillway can linearly or even non-linearly depend on the information of two or even more parts of the boundary, as in \cite{bastin2016stability,DIAGNE2012109} for one-dimensional case. 
Furthermore, the framework can be easily extended to higher dimensional problems. Then, of course, it will be more complicated to design control laws.


\end{document}